\documentclass{amsart}

\usepackage{amsmath, amsthm, amssymb}
\usepackage[all]{xy}

\newtheorem{thm}{Theorem}
\newtheorem{prop}{Proposition}
\newtheorem{cor}{Corollary}
\newtheorem{lem}{Lemma}

\newtheorem*{thmf}{Th\'eor\`eme}
\newtheorem*{lemf}{Lemme}

\theoremstyle{remark}
\newtheorem{rem}{Remark}
\newtheorem{expl}{Example}
\newtheorem*{remf}{Remarque}

\theoremstyle{definition}
\newtheorem{defn}{Definition}

\def\reg{\hbox{\rm reg}}
\def\ann{\hbox{\rm ann}}

\def\indeg{\hbox{\rm indeg}}

\def\proj{\hbox{\rm Proj}}
\def\rees{\hbox{\rm Rees}}

\def\sym{\hbox{\rm Sym}}
\def\div{\hbox{\rm Div}}
\def\syz{\hbox{\rm Syz}}
\def\tor{\hbox{\rm Tor}}
\def\fin{\hbox{\rm end}}

\def\homgr{\hbox{\rm Homgr}}

\def\B{{\mathcal{B}}}

\def\Z{{\mathcal{Z}}}

\def\homgr{\hbox{\rm Homgr}}

\def\om{{\omega}}

\def\b{{\beta}}
\def\a{{\alpha}}
\def\im{{\mathfrak{m}}}
\def\ip{{\mathfrak{p}}}

\def\cs{{\mathcal{S}}}
\def\ch{{\mathcal{H}}}
\def\ks{{\mathfrak{KS}}}
\def\ra{{\rightarrow}}
\def\ol{\overline}
\def\lra{{\longrightarrow}}
\newcommand{\surj}{\twoheadrightarrow}

\newcommand{\lr}[1]{\langle #1 \rangle}

\title{Elimination and nonlinear equations of Rees algebra}
\author[Laurent Bus\'e]{Laurent Bus\'e}
\address{INRIA Sophia Antipolis - M\'editerran\'ee.
            2004 route des Lucioles, B.P. 93,
            F-06902 Sophia Antipolis}
\email{Laurent.Buse@inria.fr}
\urladdr{\tt http://www-sop.inria.fr/members/Laurent.Buse/}

\author[Marc Chardin]{Marc Chardin}
\address{Institut de Math\'ematiques de Jussieu.
UPMC, Boite 247, 4, place Jussieu,
F-75252 PARIS CEDEX 05}
\email{chardin@math.jussieu.fr}
\urladdr{http://people.math.jussieu.fr/~chardin/}

\author[Aron Simis]{Aron Simis}
\address{Universidade Federal de Pernambuco.
Departamento de Matem\'atica,
Av. Prof. Luiz Freire, s/n,
Cidade Universit\'aria, CEP 50740-540,
Recife - Pernambuco - Brasil.}
\email{aron@dmat.ufpe.br}

\date{\today}

\begin{document}

\maketitle
\centerline{with an appendix by Joseph Oesterl\'e}

\begin{abstract}
A new approach is established to computing the image of a rational map,
whereby the use of approximation complexes is complemented with a detailed analysis of the torsion of the
symmetric algebra in certain degrees. In the case the map is everywhere defined this analysis provides free
resolutions of graded parts of the Rees algebra of the base ideal in degrees where it does not coincide with the
corresponding symmetric algebra. A surprising fact is that the torsion in those degrees
only contributes to the first free module in the resolution of the symmetric
algebra modulo torsion. An additional point is that this contribution -- which of course corresponds
to non linear equations of the Rees algebra --  can be described in these degrees
in terms of non Koszul syzygies 
via certain upgrading maps in the vein of the ones introduced earlier by
J. Herzog, the third named author and W. Vasconcelos.
As a measure of the reach of this torsion analysis we could say that, in the
case of a general everywhere defined map, half of the degrees where the torsion does not vanish
 are understood.

\end{abstract}
\section{Introduction}

Let $k$ stand for an arbitrary field, possibly assumed to be of  characteristic zero in some parts of
this work.
Let $R:=k[X_1,\ldots , X_n]\,(n\geq 2)$ denote a standard graded polynomial ring over the field $k$
and let $I\subset R$ denote an ideal generated by $k$-linearly independent forms $\mathbf{f}=\{f_0,\ldots ,f_n\}$
of the same degree $d\geq 1$.
Set $\im :=(X_1,\ldots , X_n)$.

Throughout the paper $I$ will be assumed to be of codimension at least $n-1$, i.e., that $\dim R/I\leq 1$.
In the terminology of rational maps, we are assuming that the base locus of the rational map
defined by $\mathbf{f}$ consists of a finite (possibly, empty) set of points.
Furthermore, for the purpose of elimination theory we will always assume that $\dim k[\mathbf{f}]=\dim R$, i.e.,
that the image of the rational map is a hypersurface.

The background for the contents revolves around the use of the so-called approximation complex ${\mathcal Z}$
(\cite{HeSV83}) associated to $I$ in order to extract free complexes over a polynomial ring that yield the equation
of the eliminated hypersurface, at least in principle.
This idea was originated in \cite{BuJo03} to which subsequent additions were made in \cite{BuCh05}
and \cite{BCJ}.

We note that the complex ${\mathcal Z}={\mathcal Z}(I)$ is in the present case an acyclic complex of bigraded modules
over the standard bigraded polynomial ring $S:=R[T_0,\ldots, T_n]$.
The gist of the idea has been to look at the the one-side $\mathbb{N}$-grading of $S$ given by
$S=\oplus_{\mu\geq 0}S_{\mu}$, where $S_{\mu}:= R{_\mu}\otimes_k k[T_0,\ldots, T_n]$ is naturally
a free $k[T_0,\ldots, T_n]$-module.
When ``restricted'' to this $\mathbb{N}$-grading, ${\mathcal Z}(I)$ gives a hold of the corresponding
graded pieces of the symmetric algebra $\sym_R (I)$ of $I$.
In order to set up the next stage 
one has to assume
some {\em threshold degree} beyond which
the annihilator of the graded piece of the symmetric algebra stabilizes.
The final step is to read the eliminated equation off a matrix of the presentation map of such a graded piece.

One basic question is to express this sort of threshold degree in terms of the numerical invariants
stemming from the data, i.e., from $I$.
In \cite{BuCh05} one such invariant was introduced which involved solely the integers $n,d$ and the initial
degree of the $\im$-saturation of $I$.

In the present incursion into the question we take a slight diversion by bringing up the symmetric algebra of $I$
modulo its $\im$-torsion.
In a precise way, we shift the focus to the $R$-algebra $\cs_I^* :=\sym_R (I)/H^0_\im (\sym_R (I))$.
This algebra is an intermediate homomorphic image of $\sym_R (I)$ in the way to get the Rees algebra
$\rees_R (I)$ of $I$.
In fact, when $I$ is $\im$-primary -- so to say, half of the cases we have in mind -- one has
$\cs_I^*=\rees_R (I)$.

Correspondingly, we introduce yet another threshold degree $\mu_0(I)$ involving, besides the basic integers $n,d$,
also numerical data of the Koszul homology of $I$.
All results will of this paper will deal with integers (degrees) $\mu$ satisfying $\mu\geq \mu_0(I)$ --
any such integer will be named a {\em threshold integer}.
Moreover, a dimension theoretic restriction will be assumed, namely that $\dim \sym_R (I)=\dim \rees_R (I)$.
By \cite[Proposition 8.1]{HeSV83}, this is equivalent to requiring a typical bound on the
local number of minimal generators of $I$, to wit
$$\nu (I_\ip )\leq \dim R_\ip +1, \quad \mbox{\rm for every prime ideal}\quad  \ip \supset I.$$
In the present context, this is no requirement whatsoever if $I$ happens to be $\im$-primary, and in the
codimension $n-1$ case it is imposing that $I$ be generically generated by $n=\dim R$ elements -- i.e.,
a drop by one from the global number of equations.
The need for this assumption stems clear from mimicking an almost complete intersection
of codimension $n$.

Section 2 contains the main structural result related to the
threshold degree (Theorem~\ref{thm:Fl/Fl-1free}).
Firstly, we prove the
vanishing of the graded components,
beyond the threshold degree, of all local cohomology modules (of order $i\geq 1$) of $\sym_R (I)$ with
support on $\im$; and secondly, we prove the freeness as $k[T_0,\ldots, T_n]$-modules, of the graded components,
beyond the threshold degree, of $J\lr{\ell}/J\lr{\ell -1}$, for $\ell\geq 2$, alongside with the values
of their ranks.
Here, $J:=\ker (S \surj \cs_I^*)$ with $T_j\mapsto f_j$ and $J\lr{\ell}$ denotes its degree $\ell$ homogeneous
part in the standard grading of $R[T_0,\ldots,T_n]$.

As it turns the free $k[T_0,\ldots, T_n]$-modules $(J\lr{\ell}/J\lr{\ell -1})_{\mu}$
are crucial in writing a free $k[T_0,\ldots, T_n]$-resolution of the module $(\cs_I^*)_{\mu}$
($\mu$ a threshold integer). This resolution is given in Corollary~\ref{cor:genres}.

Section 3 deals with the $\im$-primary case.
The one main result is a sharp lower bound for the threshold degree in terms of $n,d$.
This bound is actually attained in characteristic zero provided the forms ${\mathbf f}$
are general.
The proof depends on the form of the Hilbert series of a well-known $\im$-primary almost complete
intersection -- the guessed form of the series is actually not entirely obvious.
There are at least two ways of getting it, one of which a Lefschetz type of argument.
We added an appendix with a more elementary proof due to Oesterl\'e.

This bound in turn allows, by tuning up a threshold integer $\mu$, to bound the degrees
of the syzygies of $I$ that may appear in the presentation matrix of $(\cs_I^*)_{\mu}$
in the aforementioned free $k[T_0,\ldots, T_n]$-resolution.
As a consequence, the form of the resolution for such a choice of a threshold integer
becomes more explicit (see Corollary~\ref{mu0gen}).

Another piece of interest in this section is that, in the way of proving Theorem~\ref{thm:Fl/Fl-1free},
we obtain in the $\im$-primary case an isomorphism of
$k[T_0,\ldots, T_n]$-modules
$$
			(J\lr{\ell}/J\lr{\ell -1})_{\mu }\simeq (H_1)_{\mu +\ell d}\otimes_k k[T_0,\ldots, T_n](-\ell),
$$
for $\mu$ a threshold integer, $\ell\geq 2$.
We show that this isomorphism is really the expression of a so-called
\emph{downgrading map} (see Proposition~\ref{prop:downgrading}).
Versions of such maps
have been considered in \cite{HSV09} and even earlier in a slightly different form (\cite{HeSV83}).

In Section 4 we try to replay the results of the previous section when $\dim R/I=1$.
We still obtain a good lower bound for the threshold number in terms of $n,d$. The argument is different since there
is no obvious model to compare the respective Hilbert functions as in the $\im$-primary case.

Finally, Section 5 is devoted to a few examples of application in implicitization to illustrate
how the present theory works in practice.

\section{The main theorem}

Let $R:=k[X_1,\ldots , X_n]$, with $n\geq 2$, stand for the standard graded polynomial ring over a field $k$
and let $I\subset R$ denote an ideal generated by $k$-linearly independent forms $\mathbf{f}=\{f_0,\ldots ,f_n\}$
of the same degree $d\geq 1$.
Set $\im :=(X_1,\ldots , X_n)$.

Throughout it will be assumed that $\dim R/I\leq 1$.
In addition, for the purpose of implicitization, we assume that $\dim k[\mathbf{f}]=\dim R$, i.e.,
that the image of the rational map is a hypersurface.

Let $K_i:=K_i(f_0,\ldots ,f_n ;R)$ denote the term of degree $i$ of the Koszul complex associated to $\mathbf{f}$,
with $Z_i, B_i, H_i=Z_i /B_i$ standing for the module of cycles,  the  module of borders and  the homology module
in degree $i$, respectively. Since the ideal $I$ is homogeneous, these modules inherit a natural structure of graded $R$-modules.

Letting $T_0,\ldots ,T_n$ denote new variables over $k$, set $R':=k[T_0,\ldots ,T_n]$ and $S:=R\otimes_k R'
\simeq R[T_0,\ldots ,T_n]$.
Let $J$ stand for the kernel of the following graded $R$-algebra homomorphism
\begin{eqnarray*}
	S & \rightarrow & \cs_I^* :=\sym_R (I)/H^0_\im (\sym_R (I)) \\
	T_i & \mapsto & f_i
\end{eqnarray*}
The ideal $H^0_\im (\sym_R (I))$ -- which could be called the $\im$-{\em torsion} of $\sym_R (I)$ --
is contained in the full $R$-torsion of $\sym_R (I)$.
Therefore, there is a surjective a graded $R$-homomorphism onto the Rees algebra of $I$
$$\cs_I^* \surj\rees_R (I)$$
which is injective if and only if $\nu (I_\ip )=\dim R_\ip$ for every prime $\ip \supset I$
such that $\ip\neq \im$, where $\nu(\_\,)$ denotes minimal number of generators.
In particular, if $I$ has codimension $n$ (i.e., if $I$ is $\im$-primary) then $\cs_I^* $ is the Rees algebra
$\rees_R (I)$.

Given an integer $\ell\geq 0$ we consider the ideal $J\lr{\ell}\subset J$ generated by elements in $J$ whose degree in the $T_i$'s
is at most $\ell$. Thus $J\lr{0}=0$ and $J\lr{1}\simeq \syz_R (f_0,\ldots ,f_n)S$ via the identification of a syzygy $(a_0,\ldots ,a_n)$
with the linear form $a_0T_0+\ldots +a_nT_n$, and $\sym_R (I)\simeq S/J\lr{1}$.

We will denote by $\ks$ (for $\mathfrak{K}$oszul $\mathfrak{S}$yzygies) the $S$-ideal generated by
the elements $f_iT_j-f_jT_i$ ($0\leq i,j\leq n$) and set $\ol{S}:=S/\ks$. Notice that $\ks\subset J\lr{1}$ where the inclusion is strict (for the sequence $\mathbf{f}$ cannot be $R$-regular). Observe also that the module $J\lr{\ell}/J\lr{\ell -1}$ is generated exactly in degree $\ell$.

Finally, for any $\mathbb{N}$-graded module $M$, we will denote
$$\indeg (M):=\inf \{ \mu \ \vert \ M_\mu \not= 0\},$$
with the convention that $\indeg (0)=+\infty$, and
$$\fin (M):= \sup \{ \mu \ \vert \ M_\mu \not= 0 \},$$ with the convention $\fin(0)=-\infty$.

\medskip
We next introduce the basic numerical invariant of this work and give
it a name for the sake of easy reference throughout the text.

\begin{defn}\rm
The {\em threshold degree} of the ideal $I=(\mathbf{f})$ is the integer
$$
	\mu_0 (I): = (n-1)(d-1)-\min\{ \indeg (H_1(\mathbf{f} ;R)),\indeg (H^0_\im (H_1(\mathbf{f};R)))-d\}.
$$
\end{defn}
Any integer $\mu$ such that $\mu\geq \mu_0(I)$ will be likewise referred to as a {\em threshold integer}.

\medskip

Note that the threshold degree does not depend  on the choice of a minimal set of generators.

If $\dim(R/I)\leq 1$ then
$$\mu_0 (I)=\max \{  \fin (H^0_\im (R/I))-d,\fin (H^1_\im (H_1(\mathbf{f} ;R)))-2d \} +1 $$
by Koszul duality.

The threshold degree will play a key role throughout this paper and it will soon become clear why it is called this way.
Notice that whenever $I$ is $\im$-primary, then $\mu_0(I)=\reg(I)-d$, where
$$\reg(I)=\min\{ \nu \textrm{ such that } H^i_\im(I)_{>\nu-i}=0 \}$$
stands for the Castelnuovo-Mumford regularity of $I$.

Also, in the $\im$-primary case, the threshold degree is related to the numerical invariant $r(I)$ introduced in \cite[Theorem 2.14]{HSV09},
namely, one has  $r(I)+d=\mu_0(I)$. 	

The more detailed nature of  $\mu_0(I)$ will be discussed in Sections \ref{sec:mu0:m-p} and \ref{sec:mu0}.

We will hereafter consider the one-side $\mathbb{N}$-grading of $S$ given by $S=\oplus_{\mu\geq 0}S_{\mu}$,
where $S_{\mu}:= R{_\mu}\otimes_k R'$. Likewise, if $M$ is a bigraded $S$-module, then $M_\mu$ stands for the homogeneous
component of degree $\mu$ of $M$ as an $\mathbb{N}$-graded module over $S$ endowed with the one-side grading.
Note that $M_\mu$ is an $R'$-module.

Recall that our standing setup has $\dim(R/I)\leq 1$, i.e., either $I$ is $\im$-primary
or has codimension one less.
Most of the subsequent results will deal with integers (degrees) $\mu$ satisfying $\mu\geq \mu_0(I)$ --
recall that any such integer is being named a threshold integer.
Moreover, it will be assumed throughout that $\dim \sym_R (I)=\dim \rees_R (I)$, which
by \cite[Proposition 8.1]{HeSV83} is tantamount to requiring the well-known bounds
$$\nu (I_\ip )\leq \dim R_\ip +1, \quad \mbox{\rm for every prime ideal}\quad  \ip \supset I.$$

The following basic preliminary seems to have gone unnoticed.

\begin{lem}\label{lem:indeg} Let $\dim (R/I)\leq 1$.
If $\nu (I_\ip )\leq \dim R_\ip +1$ for every prime ideal $\ip \supset I$, then $\indeg(H_2)\geq \indeg(H_1)+d$.
\end{lem}
\begin{proof} First, if $\dim (R/I)=0$ then $H_2=0$ and the claimed inequality holds by convention.
Consequently, from now on we assume that $\dim (R/I)=1$.
We may assume that $k$ is an infinite field.	Let ${\bf g}:=\{g_1,\ldots ,g_n\}$ be general $k$-linear combinations of the $f_i$'s
and set $J:=({\bf g})$. Since $\nu (I_\ip )\leq \dim R_\ip +1$ for every non maximal prime ideal $\ip \supset I$,  $J$ and $I$ have the same saturation with respect to $\im$. Also, since the $f_i$'s are minimal generators of $I$, they form a $k$-basis
of the vector space $I_d$. Therefore,  $I=J+(f)$ for some $f\in I_d$, hence
$H_i\simeq H_i ({\bf g},f;R)$ for every integer $i$. Moreover, one has the exact sequence of complexes
$$0 \rightarrow K_\bullet({\bf g};R) \xrightarrow{\iota} K_\bullet({\bf g},f;R) \xrightarrow{\pi}
K_{\bullet -1}({\bf g};R)[-d] \rightarrow 0$$
where $\iota$ is the canonical inclusion and $\pi$ the canonical projection. It yields the long exact sequence
\begin{multline}\label{eq:LongExactSequence}
	\ldots \rightarrow H_2({\bf g};R) \xrightarrow{\overline{\iota_2}}  H_2 \xrightarrow{\overline{\pi_2}}
H_1({\bf g};R)[-d] \xrightarrow{\cdot (-f)} H_1({\bf g};R) \\
	\xrightarrow{\overline{\iota_1}} H_1 \xrightarrow{\overline{\pi_1}} H_0({\bf g};R)[-d]
\xrightarrow{\cdot\, f} H_0({\bf g};R) \xrightarrow{\overline{\iota_0}} H_0 \rightarrow 0
\end{multline} 	
Now, since ${\bf g}$ is an almost complete intersection $H_1({\bf g};R)$ is isomorphic to the canonical
module $\omega_{A/J}$ (see \cite[Proof of Lemma 2]{BuCh05}).
But the latter is annihilated by $J$, hence the map
$$H_1({\bf g};R)[-d] \xrightarrow{\cdot(-f)} H_1({\bf g};R)$$
is the null map. By inspecting \eqref{eq:LongExactSequence}, it follows that $H_2 \simeq H_1({\bf g};R)[-d]$
and that $H_1({\bf g};R)$ is a submodule of $H_1$. Consequently, for every $\nu$ such that $(H_1)_{\nu}=0$,
we have $(H_2)_{\nu+d}=0$.	
\end{proof}

\medskip

We now prove a vanishing result for the local cohomology modules of the modules of cycles of ${\bf f}$
in terms of the threshold degree $\mu_0(I)$.

\begin{prop}\label{vanishing_of_local_cohomology}
Let $\dim (R/I)\leq 1$.
If $\nu (I_\ip )\leq \dim R_\ip +1$ for every prime ideal $\ip \supset I$, then for every threshold integer
$\mu$ one has
\begin{equation}\label{eq:HqZp=0}
	H^q_\im (Z_p )_{\mu +pd}=0 \ \text{ for } \ p\neq q     
\end{equation}
\end{prop}
\begin{proof}
Consider the approximation complex of cycles associated to the ideal $I=(f_0,\ldots,f_n)$. We denote it by $\Z$.
By definition, we have $\Z_i =Z_i [id]\otimes_R S\{ -i\}$ so that $(\Z_i )_\mu =(Z_i)_{\mu +id}\otimes_k R'\{ -i\}$.
As proved in \cite[Theorem 4]{BuCh05}, the complex $\Z$ is acyclic under our assumptions. To take advantage of the
spectral sequences associated to the double complex $\mathcal{C}_\im^p(\Z_q)$, the knowledge of the local cohomology
of the cycles of the Koszul complex associated to $f_0,\ldots,f_n$ is helpful.
One has the following graded (degree zero) isomorphisms of $R$-modules \cite[Lemma 1]{BuCh05}:
	\begin{equation}\label{eq:isoHqZp}
	 H^q_\im (Z_p ) \simeq
	\begin{cases}
		0&\quad {\rm if}\ q<2\ {\rm or}\ q>n \\
		 H^0_\im (H_{q-p})^*[n-(n+1)d]&\quad {\rm if}\ q=2 \\
		{H_{q-p}}^*[n-(n+1)d]&\quad {\rm if}\ 2< q\leq n-1 \\
		 {Z_{n-p}}^{*}[n-(n+1)d]&\quad {\rm if}\ q=n \\
	\end{cases}		
	\end{equation}
 	 where $\hbox{---}^{*}:=\homgr_R (\hbox{---},k)$, and also \cite[Proof of Lemma 1]{BuCh05}
\begin{equation}\label{eq:isoH2Z2}
	 H^0_\im(H_0)^*[n-(n+1)d]\simeq H^2_\im(Z_2)\simeq H^0_\im(H_1)
\end{equation}
We now consider various cases.

The vanishing is obvious if $q<2$ and $q>n$.

If $q=2$ then there is only one non-trivial module
$$ H^2_\im(Z_1)\simeq H^0_\im(H_1)^*[n-(n+1)d]$$
(observe that $Z_0=S$ so that $H^2_\im(Z_0)_\mu=0$ for $\mu\geq -1$). This isomorphism shows that
$ H^2_\im(Z_1)_\mu=0$ for every
$$\mu \geq (n+1)d-n+1-\indeg(H^0_\im(H_1))=(n-1)(d-1)+2d-\indeg(H^0_\im(H_1))$$
so that $H^2_\im(Z_1)_{\mu+d}=0$ for every  $\mu\geq (n-1)(d-1)-(\indeg(H^0_\im(H_1))-d)$.

Now, assume that $3\leq q \leq n-1$. From \eqref{eq:isoHqZp} one only has to consider the modules $H_1^*$
and $H_2^*$ because $H_i=0$ if $i>2$.
By definition, $H_1^*[n-(n+1)d]_\mu=0$ for every  $\mu \geq (n-1)(d-1)+2d-\indeg(H_1)$ and similarly
$H_2^*[n-(n+1)d]_\mu=0$ for every
$\mu \geq (n-1)(d-1)+2d-\indeg(H_2)$.
It follows that $H^q_\im (Z_p )_{\mu +pd}=0$, $p\neq q$, for every
$$\mu \geq (n-1)(d-1)-\min(\indeg(H_1),\indeg(H_2)-d) \geq (n-1)(d-1)-\indeg(H_1)$$
where the last inequality holds by Lemma \ref{lem:indeg}.

Finally, if $q=n$ then
for every  $\mu \in \mathbb{Z}$ and $p<n$ we have
$$ H^n_\im (Z_p)_{\mu+pd} \simeq ({Z_{n-p}}^{*})_{\mu+n-(n-p+1)d}$$
Moreover, since $\indeg(B_{n-p})=(n-p+1)d$ the exact sequence
$$0\rightarrow B_{n-p} \rightarrow Z_{n-p} \rightarrow H_{n-p} \rightarrow 0$$
 shows that $$({Z_{n-p}}^{*})_{\mu+n-(n-p+1)d}\simeq ({H_{n-p}}^{*})_{\mu+n-(n-p+1)d}$$ for every  $\mu$
 such that $\mu+n-(n-p+1)d > - (n-p+1)d$, that is to say for every  $\mu > -n$. Therefore, we get that
 for every  $\mu > -n $ and $p<n$
$$ H^n_\im (Z_p)_{\mu+pd} \simeq ({H_{n-p}}^{*})_{\mu+n-(n-p+1)d}$$
It follows, as in the previous case, that
$H^n_\im (Z_p)_{\mu +pd} =0$ for $\mu > \mu_0(I)$ and $p<n$.
\end{proof}

\begin{thm}\label{thm:Fl/Fl-1free}
	Let $\dim (R/I)\leq 1$. If $\nu (I_\ip )\leq \dim R_\ip +1$ for every prime ideal $\ip \supset I$ then
for every threshold integer
	$\mu$, one has:
		\begin{itemize}
			\item[{\rm (i)}]  $H^{i}_\im (\sym_R (I))_\mu =0$ for $i>0$ and
			\item[{\rm (ii)}] for every $\ell\geq 2$ the $R'$-module
			$(J\lr{\ell}/J\lr{\ell -1})_\mu$ is free of rank
		\end{itemize}

	$$\left\{\begin{array}{ll}
		\dim_k (H^0_\im (H_1))_{\mu +2d} & \mbox{\rm if}\quad	\ell =2 \\[5pt]
		\dim_k (H^0_\im (H_1))_{\mu +\ell d}+\dim_k (R/I^{sat})_{(n+1-\ell )d-n-\mu  } & \mbox{\rm otherwise}.
		\end{array}\right.
$$
\end{thm}

\begin{proof}
By definition of the $\Z$-complex, $H^q_\im (Z_p )_{\mu +pd}=H_\im^q(\Z_p)_\mu$ for any $\mu$.
Fix an integer
		 $\mu\geq \mu_0 (I)$. By (\ref{eq:HqZp=0}),  the spectral sequence
$H^q_\im (\Z_p )_\mu \Rightarrow H^{q-p}_\im (\sym_R (I))_\mu$  implies
		 that $H^{i}_\im (\sym_R (I))_\mu =0$ for $i>0$, which proves (i), also providing
		 a filtration of the $R'$-module $H^0_\im (\sym_R (I))_\mu $
		\begin{equation}\label{filtration_from_spectral}
		0=F_0=F_1\subset F_2 \subset \cdots \subset F_t =H^0_\im (\sym_R (I)_\mu )=(J/J\lr{1})_\mu
		\end{equation}
	such that, by \eqref{eq:isoH2Z2},
	\begin{eqnarray}\label{eq:F2}
		F_2&\simeq & H^2_\im(\Z_2)_\mu \simeq H^0_\im(H_0)^*[n-(n+1)d]_{\mu+2d} \otimes_k R'\{ -2 \}\nonumber\\
&\simeq & H^0_\im (H_1)_{\mu + 2d}\otimes_k R'\{ -2 \}.
\end{eqnarray}
Clearly, \eqref{eq:F2} shows  that $F_2$ is a finite free $R'$-module of rank
$\dim_k (H^0_\im (H_1))_{\mu +2d}$.	
	
By a similar token, for every  $\ell \geq 3$,
	\begin{equation}\label{eq:Fl/Fl-1}
		F_\ell /F_{\ell-1}\simeq H^\ell_\im(\Z_\ell)_\mu  \simeq {H_0}^*[n-(n+1)d]_{\mu+ \ell d}\otimes_k R'\{ -\ell \}	
	\end{equation}
	 In particular, $F_\ell /F_{\ell-1}$ is a free $R'$-module which is generated in degree $\ell$.
	Now, as $(J\lr{\ell}/J\lr{\ell -1})_\mu$ is also generated in degree $\ell$, an easy recursive argument on $\ell \geq 1$
	yields $F_\ell =(J\lr{\ell}/J\lr{1})_\mu$.

Therefore $(J\lr{2}/J\lr{1})_\mu$ is a free $R'$-module of rank
$\dim_k (H^0_\im (H_1))_{\mu +2d}$, which shows the case $\ell=2$ of item (ii).

 To get the case $\ell\geq 3$, recall that $H^1_\im(Z_2)=0$.  According to \cite[Equation (3) in Proof 3]{BuCh05},
 there is an exact sequence of graded $R$-modules
\begin{equation}\label{eq:HO*-H2Z2}
0 \rightarrow  H^1_\im(H_2) \rightarrow {H_0}^*[n-(n+1)d] \rightarrow H^2_\im(Z_2) \rightarrow 0	
\end{equation}
Moreover, $H^2_\im(Z_2)\simeq H^0_\im(H_1)$ by \eqref{eq:isoH2Z2} and by local duality
$$H^1_\im(H_2) \simeq (R/I^{\rm sat})^*[n-(n+1)d].$$
Therefore, we deduce
that if $\ell \geq 3$ then $F_\ell /F_{\ell-1}\simeq (J\lr{\ell}/J\lr{\ell -1})_\mu$ is a finite
free $R'$-module of rank
$$ \dim_k ((H_0)^*[n-(n+1)d])_{\mu+ \ell d} = \dim_k (H^0_\im (H_1))_{\mu +\ell d}+\dim_k (R/I^{\rm sat})_{(n+1-\ell )d-n-\mu  }.$$
This finishes the proof.
\end{proof}

\begin{cor}\label{cor:genres} Let $\dim (R/I)\leq 1$. If $\nu (I_\ip )\leq \dim R_\ip+1$ for every
prime ideal $\ip \supset I$ then, for every  threshold integer $\mu$,
			the  $R'$-module $(\cs_I^*)_\mu$ admits a minimal graded free $R'$-resolution of the form
			\begin{equation}\label{eq:resolutionSI*}
				\cdots \ra (\Z_i)_\mu \ra \cdots \ra (\Z_2)_\mu \ra (\Z_1)_\mu
				\oplus_{\ell=2}^{n}(J\lr{\ell}/J\lr{\ell -1})_{\mu} \ra
				(\Z_0)_\mu=R_\mu\otimes_k R' .
			\end{equation}
		\end{cor}
		\begin{proof}
			By Theorem~\ref{thm:Fl/Fl-1free}, $(J\lr{\ell}/J\lr{\ell -1})_\mu $ is a free $R'$-module for $\ell \geq 2$. Therefore
			the split exact sequence
				$$0\ra J\lr{\ell -1}_\mu \ra J\lr{\ell}_\mu \ra (J\lr{\ell}/J\lr{\ell -1})_\mu \ra 0$$
				gives rise to an exact sequence
				$$
				0\ra J\lr{\ell -1}_\mu \otimes_{R'} k\ra J\lr{\ell}_\mu \otimes_{R'} k\ra (J\lr{\ell}/J\lr{\ell -1})_\mu \otimes_{R'} k\ra 0,		 $$
				which shows that the minimal generators of $J\lr{\ell}_\mu$ are the minimal generators of $J\lr{\ell -1}_\mu$ plus the generators of $(J\lr{\ell}/J\lr{\ell -1})_\mu$.
				Moreover, we also deduce the isomorphisms $\tor_i^{R'}(J\lr{\ell}_\mu ,k)\simeq \tor_i^{R'}(J\lr{\ell -1}_\mu ,k)$ for $i>0$ that imply by induction that $\tor_i^{R'}(J_\mu ,k)\simeq \tor_i^{R'}(J\lr{1}_\mu ,k)$ for $i>0$. It follows that the complex \eqref{eq:resolutionSI*}, which is built by adding to the complex $(\Z_\bullet)_\mu$ the canonical map
				$$\oplus_{\ell=2}^{n}(J\lr{\ell}/J\lr{\ell -1})_{\mu} \ra
				R_\mu\otimes_k R'$$
is acyclic.
		\end{proof}
		
		\begin{cor}
			Let $\dim (R/I)\leq 1$. If $\nu (I_\ip )\leq \dim R_\ip$ for every
			non maximal prime ideal $\ip \supset I$ then, for every  threshold integer $\mu$,
						the  $R'$-module $\rees_R (I)_\mu$ admits \eqref{eq:resolutionSI*} as a minimal graded free $R'$-resolution
		\end{cor}
\begin{proof}
	This follows from the well known property that
	$\cs_I^*=\rees_R (I)$ if and only if $\nu (I_\ip )=\dim R_\ip$ for every  non maximal prime ideal $\ip \supset I$.	 
\end{proof}

\section{The $\im$-primary case}

In this section, we will concentrate on the case where the ideal $I$ is  $\im$-primary.
As we will see, in such a situation it is remarkable that all the syzygies of $I$ used in the
matrix-based representation of a hypersurface $\ch$ obtained by this method can be recovered from the linear
syzygies via downgrading maps; this is described in Section \ref{downgrading}.

Recall that in the $\im$-primary case, $H^1_\im(H_1)=0$, $H^0_\im(H_1)=H_1$ and we have
$$\mu_0=\reg(I)-d=\fin(H^0_\im(R/I))-d+1= n(d-1) - \indeg(H_1)+1$$


\subsection{Bounds for the threshold degree}\label{sec:mu0:m-p}

It is clear that $\mu_0(I)\leq (n-1)(d-1)$. In this section, we will give a sharp lower bound for $\mu_0(I)$.
We begin with a preliminary result.

 \begin{prop}
	We have
	$$
	\fin (R/I)\geq \left\lfloor {{(n+1)(d-1)}\over{2}}\right\rfloor
	$$
	and equality holds if the forms $f_0,\ldots ,f_n$ are sufficiently general, and $k$ has characteristic $0$.
 \end{prop}

\begin{proof}{\rm [Iarrobino-Stanley; see also the Appendix at the end of this paper]}
By the K\"unneth formula, $X:=({\bf P}^{d-1})^n$ has de Rham cohomology ring	
isomorphic to	
$$
	C={\bf Z}[\om_1,\ldots ,\om_n]/(\om_1^d,\ldots ,\om_n^d),	
$$	
where $\om_i$ is a K\"ahler form on the
$i$-th factor ${\bf P}^{d-1}$ and the cup product in the cohomology ring corresponds to the usual product in $C$.
The form $\om :=\om_1+\cdots +\om_n$ is a K\"ahler form on $X$ and
the Hard Lefschetz Theorem applied to $\om$ shows that for $0\leq \mu <n(d-1)/2$, the multiplication by $\om$ in $C_\mu$
is injective and the multiplication by $\om^{n(d-1)-2\mu}$ induces the Poincar\'e duality from
$C_\mu \otimes_{\bf Z} {\bf Q}$ to $C_{n(d-1)-\mu}\otimes_{\bf Z} {\bf Q}$. It follows that for $\mu \leq \nu$,
multiplication by $\om^{\nu -\mu}$ in $C\otimes_{\bf Z} {\bf Q}$ is injective if
$\dim C_\mu \leq \dim C_\nu$ and onto  if $\dim C_\mu \geq \dim C_\nu$.

This shows that the Hilbert function of $B:={\bf Q}[\om_1,\ldots ,\om_n]/(\om_1^d,\ldots ,\om_n^d,\om^d )$ is $h(i):=\max \{ 0,a_i\}$,
where $a_i$ is given by
$$
	{{(1-t^d)^{n+1}}\over{(1-t)^n}}=\sum_{i=0}^{(n+1)d-n}a_i t^i.
$$
	As $a_i=-a_{(n+1)d-n-i}$, $a_i>0$ if and only if $0<i\leq \left\lfloor {{(n+1)(d-1)}\over{2}}\right\rfloor$, hence $\fin (B)= \left\lfloor {{(n+1)(d-1)}\over{2}}\right\rfloor$. Now notice that if
	the $f_i$'s are general forms of degree $d$, then $n$ of them, say $f_1,\ldots ,f_n$, form a regular sequence. The quotient $A$ by this
	regular sequence has the same Hilbert function as $C$ and therefore the Hilbert function of $R/I$ is bounded below by $h(i):=\max\{ 0,\dim_k A_i-\dim_k A_{i-d}\}$. As this lower bound is reached by $B$, it is reached on a non empty Zariski open subset of the coefficients of $n+1$ forms of degree
	$d$, if the field contains ${\bf Q}$.
\end{proof}

\begin{cor}\label{mu0gen} If $I$ is $\im$-primary, then the threshold degree satisfies the inequalities
	$$\left\lfloor {{(n-1)(d-1)}\over{2}}\right\rfloor \leq \mu_0(I) \leq (n-1)(d-1)$$
	and equality on the left holds if the forms $f_0,\ldots ,f_n$ are sufficiently general, and $k$ has characteristic $0$.
\end{cor}
\begin{proof} If the forms $f_0,\ldots ,f_n$ are sufficiently general, then the ideal $I$ is $\im$-primary and it follows that $H^0_\im(H_i)=H_i$, $i=0,1$. Therefore
	\begin{multline*}
		\mu_0(I)=(n-1)(d-1) - \indeg(H_1) + d = \\
		(n-1)(d-1)-(n+1)d+n+ \left\lfloor {{(n+1)(d-1)}\over{2}}\right\rfloor + d
	\end{multline*}
	where the second equality follows by Koszul duality, and the stated formula follows from a straightforward computation. 	
\end{proof}

\subsection{Free resolutions}

Let $M_\mu$ denote the matrix of the presentation map  in Corollary~\ref{cor:genres}
$$(\Z_1)_\mu
\oplus_{\ell=2}^{n}(J\lr{\ell}/J\lr{\ell -1})_{\mu} \ra
(\Z_0)_\mu=R_\mu\otimes_k R,$$
with respect to a fixed, but otherwise arbitrary, basis.

The following result gives the degree of the syzygies that appears in this matrix. This degree depends on the choice of the integer $\mu$.

\begin{prop}\label{prop2} Assume that $I$ is $\im$-primary and let $\mu$ denote a threshold integer.
Suppose that $\ell$ is an integer such that some syzygy of the power $I^\ell$ appears in the matrix $M_\mu$ as defined above.
Then
	$$\ell\leq \left\lceil {{\indeg (H_1)}\over{d}}\right\rceil\leq \left\lceil {{n+1}\over{2}}\right\rceil.$$
\end{prop}
\begin{proof} By Theorem \ref{thm:Fl/Fl-1free}, if $\mu +(\ell +1)d>\fin (H_1)$, then $M_\mu$ involves syzygies of $I^j$ with $1\leq j\leq\ell$ for $\mu\geq \mu_0$. Now $\fin (H_1)=(n+1)d-n$ and the equation
$$
(n-1)(d-1)-\indeg (H_1)+d+(\ell +1)d\geq (n+1)d-n+1
$$
can be rewritten $\ell d\geq \indeg (H_1)$.
By Corollary \ref{mu0gen} it obtains
$$\indeg(H_1)\leq (n-1)(d-1)+d-\left\lfloor {{(n-1)(d-1)}\over{2}}\right\rfloor <{{(n+1)d}\over{2}},$$
which shows the second estimate.
\end{proof}

The above proposition also shows that one can tune a threshold integer $\mu$ in order to bound the degree
of the syzygies of $I$ that may appear in the matrix $M_\mu$. More precisely, choose an integer $l\in \{1,\ldots, \left\lceil {{n+1}\over{2}}\right\rceil\}$.
Then, $M_\mu$ involves only syzygies of $I$ of degree at most $l$ for every
$$\mu \geq \max\{ (n-1)(d-1)-(l-1)d,  \mu_0(I) \} = \max\{ (n-l)(d-1)-(l-1),  \mu_0(I) \}$$
For instance, $M_\mu$ involves only linear syzygies of $I$ (i.e.~$l=1$) for every  $\mu\geq (n-1)(d-1)$, and it involves only linear and
quadratic syzygies of $I$ for every  $\mu \geq \max\{(n-2)(d-1)-1,\mu_0(I)\}$, and so forth.

 \begin{cor}\label{cor:resolution} If the forms $f_0,\ldots ,f_n$ define an $\im$-primary ideal then for any threshold integer $\mu $
	the graded $R'$-module  $(\cs_I^*)_\mu$ has a minimal graded free $R'$-resolution of the form
\begin{equation}\label{eq:resolutionSI*-gen}
	\cdots \ra R'\{ -i \}^{b_i} \ra \cdots \ra R'\{ -2 \}^{b_2} \ra R'\{ -1 \}^{b_1}
	\oplus_{\ell=1}^{\left\lceil {{n+1}\over{2}}\right\rceil }R'\{ -\ell \}^{\b_\ell} \ra
	{R'}^{{\mu +n-1}\choose{n-1}}
\end{equation}
	with $\b_\ell :=\dim_k (H_0)_{(n+1-\ell )d-n-\mu  }$ and
	$$b_i:=\dim_k (B_i)_{\mu +id}= \sum_{k=1}^{\min\{ n-i,\left\lfloor {{\mu }\over{d}}\right\rfloor \}} (-1)^{k+1}
	{{\mu -kd+n-1}\choose{n-1}}^{{n}\choose{i+k}}$$
	
\end{cor}
\begin{proof}
	$\ks$ is resolved by the back part of the
	 $\Z$-complex and, as $H_i=0$ for $i\geq 2$, this coincides
	 with the corresponding part of the $\B$-complex.
\end{proof}

It is interesting to describe explicitly the maps involved in the resolution \eqref{eq:resolutionSI*-gen}.
\begin{itemize}
	\item The map $R'\{ -i \}^{b_i} \ra R'\{ -i+1 \}^{b_{i-1}}$ is the degree $\mu$ part of the $i$-th map in the $\B$-complex.
	It is given by $(B_i)_{\mu +id}\otimes_k R'\{ -i\}{\buildrel{d_i^{T}}\over{\lra}}(B_{i-1})_{\mu +(i-1)d}\otimes_k R'\{ -i+1\}$ (composed with the inclusion $(B_0)_\mu \otimes_k R' =I_\mu \otimes_k R'\subset R_\mu \otimes_k R'$ for $i=1$).
	\item The matrix of the map $R'\{ -1\}^{\b_1 } \ra {R'}^{{\mu +n-1}\choose{n-1}}=R_\mu \otimes_k R'$ is given by pre-images in $S_{\mu ,1}$ of a
	basis over $k$ of $(H_1)_{\mu +d}$ (the syzygies $\sum_i a_i T_i$ with $a_i$ of degree $\mu$, modulo the Koszul syzygies).
	\item The matrix of the map $R'\{ -\ell \}^{\b_\ell } \ra {R'}^{{\mu +n-1}\choose{n-1}}=R_\mu \otimes_k R'$ for $\ell \geq 2$ is given by pre-images in $S_{\mu ,\ell }$ of a
	basis of $(J\lr{\ell}/J\lr{\ell -1})_{\mu ,\ell}$ over $k$
\end{itemize}

\subsection{Downgrading maps}\label{downgrading}

		Assuming that $I$ is $\im$-primary (notice that in this case $J=(J\lr{1}:\im^\infty)$), the homology
modules $H_0$ and $H_1$ are supported on $V(I)=V(\im)$ so that \eqref{eq:HO*-H2Z2} shows that
		$${H_0}^*[n-(n+1)d]\simeq  {H^0_\im(H_0)}^*[n-(n+1)d] \simeq H^2_\im(Z_2)\simeq H^0_\im(H_1)=H_1$$
		Therefore, \eqref{eq:F2} and \eqref{eq:Fl/Fl-1} imply that for every  $\ell\geq 2$ and every  threshold
integer $\mu $ we have a graded isomorphism of $R'$-modules
		\begin{equation}\label{eq:JlJl-1=H1}
			(J\lr{\ell}/J\lr{\ell -1})_{\mu }\simeq (H_1)_{\mu +\ell d}\otimes_k R'\{-\ell\}
		\end{equation}
		The purpose of this section is to show that \eqref{eq:JlJl-1=H1} is realized by a
\emph{downgrading map} -- versions of such maps
have been considered in \cite{HSV09} and even earlier in a slightly different form (\cite{HeSV83}).

Namely,  define the map $\delta : S/\ks \rightarrow  (S/\ks )[d]\{-1\}$ by
		\begin{eqnarray*}
			\delta_p : (S/\ks )_{p} & \rightarrow & (S/\ks )_{p-1}[d] \\
			\sum_{0\leq i_1,\ldots,i_p\leq n} c_{i_1,\ldots,i_p}T_{i_1}\ldots T_{i_p} & \mapsto &
			\sum_{0\leq i_1,\ldots,i_p\leq n} c_{i_1,\ldots,i_p}f_{i_1}T_{i_2}\ldots T_{i_p}			
		\end{eqnarray*}
Note that this map is well defined. In addition, it induces for any integer $\mu\geq 0$ a (well-defined) homogeneous map of
graded $R'$-modules
		$$
		{\lambda_2^\mu}  : (J\lr{2}/J\lr{1})_\mu \rightarrow (J\lr{1}/\ks)\{-1\}_{\mu+d},
		$$
		and, for any integers $\mu \geq 0$ and $p\geq 3$, a (well-defined) homogeneous map of
graded $R'$-modules
		$${\lambda_p^\mu}:
		(J\lr{p}/J\lr{p-1})_\mu \rightarrow (J\lr{p-1}/J\lr{p-2})\{-1\}_{\mu+d} .
		$$

		\begin{lem}\label{lem:lambda2iso}
			The map ${\lambda_2^\mu}$ is injective for every $\mu \geq 0$ and is
			surjective for every  threshold integer $\mu$.
		\end{lem}

		\begin{proof} By definition, $(J\lr{1}/\ks)_{\nu}\simeq (H_1)_{\nu+d}\otimes_k R'\{-1\}$ for every integer $\nu\geq 0$.
Therefore, since ${\lambda_2^\mu}$ is a graded map and since \eqref{eq:JlJl-1=H1} is a graded isomorphism,
it suffices to show that ${\lambda_2^\mu}$ is injective for every  $\mu \geq 0$.

			Let $\a :=\sum_{0\leq i,j \leq n} c_{i,j}T_{i}T_{j} \in \ker (\lambda_2^\mu)$. By a standard property of Koszul syzygies,
there exists a skew-symmetric matrix $(a_{i,j})_{0\leq i,j \leq n}$ with entries in $R_\mu$ such that
			$$\sum_{0\leq i,j \leq n} c_{i,j}f_{i}T_{j}=\sum_{0\leq i,j \leq n} a_{i,j}f_{i}T_{j}$$
		in $R_{\mu+d}\otimes_k R'$. It follows that, for every  $j=0,\ldots,n$, $\sum_{0\leq i\leq n} (c_{i,j}-a_{i,j})f_{i}=0$,
i.e.,  $\sum_{0\leq i\leq n} (c_{i,j}-a_{i,j})T_{i}\in J\lr{1}, \ \ j=0,\ldots,n$. Thus
		\begin{equation}\label{eq:(c-a)J1}
			\sum_{0\leq i,j \leq n} (c_{i,j}-a_{i,j})T_{i}T_j\in J\lr{1}, \ \ j=0,\ldots,n.	
		\end{equation}
		Therefore, since $\sum_{0\leq i,j \leq n} a_{i,j}T_{i}T_{j} \in J\lr{1}$, one has $\a \in J\lr{1}$, as was to be shown.
		\end{proof}

\begin{rem} The above result can also be deduced from \cite[Lemma 2.11]{HSV09} and \cite[Theorem 2.14]{HSV09}.	
\end{rem}

		\begin{prop}\label{prop:downgrading}
			For every integer $p\geq 2$ and every threshold integer $\mu$ the map ${\lambda_p^\mu}$ is an isomorphism.
		\end{prop}

		\begin{proof}
			By the same token, according to \eqref{eq:JlJl-1=H1}, it suffices to show that ${\lambda_p^\mu}$ is injective
for every  $p\geq 2$ and every  $\mu\geq \mu_0(I)$. By Lemma \ref{lem:lambda2iso}, the claim holds for $p=2$.

Now assume that $p\geq 3$.
If $p=3$ pick an element
			$$\a =\sum_{i=0}^n B_iT_i\in (J\lr{3})_{\mu,3}$$
			 such that  $\lambda_3^\mu (\overline{\a} )=0$. It follows that
			 $$\delta_3^\mu ({\a} )=\sum_{i=0}^n \delta_{2}^\mu (B_i)T_i\in (J\lr{1})_{\mu+d,2}$$
			By a similar argument to the one employed in the proof of Lemma \ref{lem:lambda2iso}, it follows that  $\delta_{2}^\mu (B_i)\in (J\lr{1}/\ks)_{\mu+d,1}$ for every  $i=0,\ldots,n$. Therefore, since $\lambda_{2}^\mu$ is an isomorphism, it follows that $B_i \in (J\lr{2}/J\lr{1})_{\mu,2}$ and hence that $\overline{\a}=0$ in $(J\lr{3}/J\lr{2})_{\mu}$.

			Now, we proceed by induction on the integer $p\geq 2$ and assume that $p>3$. Pick an element
			$$\a =\sum_{i=0}^n B_iT_i\in (J\lr{p})_{\mu,p}$$
			 such that $\lambda_p^\mu (\overline{\a} )=0$, it follows that
			 $$\delta_p^\mu ({\a} )=\sum_i \delta_{p-1}^\mu (B_i)T_i\in (J\lr{p-2})_{\mu+d,p-1}$$
		By Theorem \ref{thm:Fl/Fl-1free}, $(J\lr{p-2}/J\lr{p-3})_{\mu +d}$ is a free graded $R'$-module which is generated in degree $p-2$. Therefore, we deduce that $\delta_{p-1}^\mu (B_i)\in (J\lr{p-2}/J\lr{p-3})_{\mu+d,p-2}$ for every  $i=0,\ldots,n$. Since $\lambda_{p-1}^\mu$ is an isomorphism by our inductive hypothesis, it follows that $B_i \in (J\lr{p-1}/J\lr{p-2})_{\mu,p-1}$ and hence that ${\a}=0$ in $(J\lr{p}/J\lr{p-1})_{\mu}$.
		\end{proof}

\section{The one-dimensional case}

In this section, we go back to the general situation where $\dim (R/I)\leq 1$ and will no longer
assume that the ideal $I$ is $\im$-primary. This more general class of ideals have interesting applications
to the implicitization of surfaces in a projective space defined by a parametrization whose base locus is a finite set
of points.

\subsection{Bounds for the threshold degree}\label{sec:mu0}

Recall the definition of the threshold degree:
$$
\mu_0 (I):=(n-1)(d-1)-\min\{ \indeg (H_1),\indeg (H^0_\im (H_1))-d\}.
$$
Following \cite{BuCh05}, one sets
$$\nu_0(I):=(n-1)(d-1) - \indeg(I^{\rm sat})$$

\begin{prop}\label{prop:regmu} Let $\dim (R/I)\leq 1$. If $\nu (I_\ip )\leq \dim R_\ip +1$ for every prime ideal $\ip \supset I$ then
	$$\reg (R/I)-d \leq \mu_0 (I) \leq \nu_0(I) \leq (n-1)(d-1)$$
for $n\geq 3$ and $d\geq 2$.
\end{prop}
\begin{proof} First, notice that since $0\leq \indeg(I^{\rm sat})\leq d$, it is clear that $0\leq \nu_0(I) \leq (n-1)(d-1)$.
	
	Now, since $\indeg(H_1)\geq d$ and $\indeg(I^{\rm sat})\leq d$, if one proves that
	\begin{equation}\label{eq:mu-inequal1}
		\indeg(H^0_\im(H_1))\geq d + \indeg(I^{\rm sat})		
	\end{equation}
	 then the inequality $\mu_0 (I) \leq \nu_0(I)$ will follow. We may assume that $k$ is an infinite field.
Let $g_1,\ldots ,g_n$ be general $k$-linear combinations of the $f_i$'s and set $J:=(g_1,\ldots ,g_n)$. The ideals $I$ and $J$ have the same
saturation with respect to $\im$ and hence $\indeg(I^{sat})=\indeg(J^{sat})$. Moreover, since $I$ is minimally generated we necessarily
have $J\subsetneq J^{\rm sat}$ (for $I^{\rm sat}=J^{sat}=J\subset I \subset I^{\rm sat}$ implies that $I=J$). As shown in the proof of Lemma \ref{lem:indeg},
one has an exact sequence
			$$
		0\ra\, H_1(g_1,\ldots,g_n;R)\ra \,H_1\ra\, 0:_{R/J}(f_i)[-d]\ra\, 0.
		$$
Therefore, since $H^0_\im(H_1(g_1,\ldots,g_n;R))\simeq H^0_\im(H_2)=0$, e.g., by \eqref{eq:isoHqZp}, it obtains
		\begin{align*}
			\indeg (H^0_\im (H_1)) & \geq \indeg (H^0_\im (0:_{R/J}(f_i))[-d]) \\
			&= d+\indeg (0:_{H^0_\im (R/J)}(f_i))\\
			&\geq d+\indeg (H^0_\im (R/J)) \\
			&= d+\indeg (J^{\rm sat}/J) \\
			& \geq d+\indeg (J^{\rm sat}) = d+\indeg (I^{\rm sat})
		\end{align*}
		
By Theorem \ref{thm:Fl/Fl-1free}, $H^i_\im (\sym_R^1 (I))_\mu =H^i_\im (I)_{\mu +d}\simeq H^{i-1}_\im (R/I)_{\mu +d}=0$
for $\mu \geq \mu_0$ and $i>0$. This proves that $\reg (R/I)\leq \mu_0 (I)+d$, as $H^i_\im (R/I)=0$ for $i>1$.
		\end{proof}

In addition to the above result, it is also possible to provide a lower bound for the threshold degree solely  in terms of $n$ and $d$.
For this purpose, we begin with a technical result.

		\begin{lem}\label{lem:ci<0}
			Let $n\geq 2$ and $d\geq 2$ be two integers and consider the polynomial
			$$\frac{(1-t^d)^n}{(1-t)^{n-1}}=\sum_{i=0}^{n(d-1)+1} c_i t^i \in \mathbb{Z}[t]$$
			Then, we have
			\begin{itemize}
				\item $c_i>0$ for every  $0\leq i \leq \lfloor \frac{n(d-1)}{2} \rfloor$,
				\item $c_i<0$ for every  $\lceil \frac{n(d-1)}{2} +1 \rceil  \leq i \leq n(d-1)+1$,
				\item if $n(d-1)+1$ is even then $c_{\frac{n(d-1)+1}{2}}=0$.
			\end{itemize}
		\end{lem}
		\begin{proof} First, observe that
			$$\frac{(1-t^d)^n}{(1-t)^{n-1}}=(1-t)\left(\frac{1-t^d}{1-t}\right)^n$$
		Now, the coefficients of the polynomial	
		$$\left(\frac{1-t^d}{1-t}\right)^n=\sum_{i=0}^{n(d-1)} d_i t^i$$
		rank along
a symmetric sequence that increases up to index $\frac{n(d-1)}{2}$, which corresponds to two indexes
if $n(d-1)$ is odd) and then decreases \cite[Theorem 1]{RRR91}. Multiplying out by $1-t$ leads to the sequence of coefficients $(d_i-d_{i-1})_i$ from which the result follows easily.
		\end{proof}

		\begin{prop}\label{notmprim} Let $\dim (R/I)\leq 1$. If $\nu (I_\ip )\leq \dim R_\ip +1$ for every prime ideal
$\ip \supset I$ then
			$$\mu_0(I)\geq \left\lfloor{{(n-2)(d-1)-1}\over{2}}\right\rfloor$$
			for $n\geq 3$ and $d\geq 2$.
		\end{prop}
		\begin{proof} We assume that $H_2\neq 0$ for otherwise Corollary \ref{mu0gen} applies.
Recalling that $K_i$ denotes the $i$th Koszul homology module of $I$, we have the
following well-known formula in $\mathbb{Z}[t]$:
			\begin{eqnarray*}
						\frac{(1-t^d)^{n+1}}{(1-t)^n}&=&
						\sum_{\mu=0}^{n(d-1)+d}\left(   \sum_{i\geq 0} (-1)^i \dim_k (K_i)_\mu   \right)t^\mu \\
						&=&\sum_{\mu=0}^{(n+1)(d-1)}\left(   \sum_{i\geq 0} (-1)^i \dim_k (H_i)_\mu   \right)t^\mu			 
			\end{eqnarray*}
But since $H_i=0$ for $i\geq 3$, the above simplifies to
			$$ \sum_{\mu=0}^{n(d-1)+d}\left( h_0(\mu)-h_1(\mu)+h_2(\mu) \right) t^\mu
=\frac{(1-t^d)^{n+1}}{(1-t)^n} \in \mathbb{Z}[t],$$
where we have set $h_i(\mu):=\dim_k (H_i)_\mu$, for $i=0,1,2$.			
			
Next consider the difference operator $\Delta h_i$ acting by $\Delta h_i (\mu):= h_i(\mu)-h_i(\mu-1)$, for $i=0,1,2$.
It follows that
\begin{eqnarray*}  \sum_{\mu=0}^{^{n(d-1)+d+1}}\left( \Delta h_0(\mu)- \Delta h_1(\mu)+ \Delta h_2(\mu) \right) t^\mu &=& (1-t)\frac{(1-t^d)^{n+1}}{(1-t)^n}\\
		&=& (1-t^d)\frac{(1-t^d)^n}{(1-t)^{n-1}}.
\end{eqnarray*}	
Applying Lemma \ref{lem:ci<0}, we find that
$$ 	\Delta h_0(\mu)- \Delta h_1(\mu)+ \Delta h_2(\mu) = c_\mu -c_{\mu-d} $$
is non positive ($<0$) for every
\begin{equation}\label{eq:mu}
	\left\lceil \frac{n(d-1)+1}{2} \right\rceil  \leq \mu \leq  \left\lfloor \frac{n(d-1)+1}{2} \right\rfloor +d.
\end{equation}
Note that $\Delta h_2(\nu) \in \mathbb{N}$ for all $\nu$ since $H^0_\im(H_2)=0$.	Therefore, for every  integer $\mu$ satisfying \eqref{eq:mu} and such that $\Delta h_0(\mu)\geq 0$ we have $\Delta h_1(\mu)>0$ and hence $h_1(\mu) \neq 0$.

The condition $\Delta h_0(\mu)\geq 0$ is fulfilled  when $H^0_\im(H_0)_{\mu-1}=0$, that is to say,
when $\fin(H^0_\im(H_0))\leq \mu-2$ or, still equivalently, when
		$$\indeg(H^0_\im(H_1))\geq (n+1)d-n-\mu+2.$$

		These considerations, applied to the lowest possible value of $\mu$ satisfying \eqref{eq:mu},
so we claim now, imply that
		\begin{equation}\label{min-indeg}
			\min \{\indeg(H_1),\indeg(H^0_\im(H_1))-d\} \leq \left\lceil \frac{n(d-1)+1}{2} \right\rceil,
		\end{equation}
from which the required lower bound follows by the definition of $\mu_0(I)$.
		
To see why (\ref{min-indeg}) holds, note that we proved the inequality $\indeg(H_1)\leq \left\lceil \frac{n(d-1)+1}{2} \right\rceil$
provided
			\begin{align*}
			\indeg(H^0_\im(H_1)) &\geq (n+1)d-n-\left\lceil \frac{n(d-1)+1}{2} \right\rceil+2=	
\left\lfloor \frac{n(d-1)+1}{2} \right\rfloor +d +1.
			\end{align*}
Thus, negating the latter inequality yields
		$$\indeg(H^0_\im(H_1)) \leq \left\lfloor \frac{n(d-1)+1}{2} \right\rfloor +d $$
		and \eqref{min-indeg} follows.
		\end{proof}

		\begin{expl}\label{HBexpl} Let $n=3$. Let $I$ denote the ideal generated by the $3$-minors of a matrix of general forms
		$\oplus_{i=1}^3 R(-e_i)\rightarrow R^4$ where $\sum_{i=1}^3 e_i=d$. As is well-known, $I$ is a codimension $2$ saturated
ideal, hence $H^0_\im(H_1)=0$ and $I$ is not $\im$-primary. By \cite[\S 1]{AH}, the module $H_1$
is generated in degree $d+\min_i\{e_i\}$. We deduce that
		$$ \left\lceil \frac{2d}{3} \right\rceil -2 \leq \mu_0(I)=d-2-\min_i\{e_i\} \leq d-3$$
		Also, we obtain that $\reg (R/I)-d=\min_i\{e_i\}-2$ and $\nu_0(I)=d-2$, which is coherent with Proposition \ref{prop:regmu}.
Notice also that if the lower bound given in Proposition~\ref{notmprim} is satisfied, the one given in Corollary~\ref{mu0gen} is not.
This shows that the assumption that $I$ be $\im$-primary in Corollary~\ref{mu0gen} is not superfluous.
		\end{expl}


\section{Application to the hypersurface implicitization problem}

Given a parametrization
\begin{eqnarray*}
\mathbb{P}^{n-1} & \rightarrow & \mathbb{P}^n \\
(X_1:\cdots:X_n) & \mapsto & (f_0:\cdots:f_{n})(X_1:\cdots:X_n) 	
\end{eqnarray*}
of a rational hypersurface $\ch$, the approximation complex of cycles associated to the ideal $I=(f_0,\ldots,f_{n})$ has been used
(see e.g.~\cite{BuJo03,BuCh05}) to derive a matrix-based representation of $\ch$. Such a representation only uses the linear syzygies
of the ideal $I$. The results obtained in the previous sections allow to produce new matrix-based representations of $\ch$ that involve
 not only the linear syzygies but also higher order syzygies of the ideal $I$. Indeed, the following proposition shows  that the divisor
 associated to $(\cs_I^*)_\mu$ has the expected property for every threshold integer $\mu$.

 \begin{prop}\label{prop:divmu0} Let $\dim (R/I)\leq 1$. If $\nu (I_\ip )\leq \dim R_\ip +1$ for every prime ideal $\ip \supset I$
then for every  $\mu \geq 0$, one has $\ann_{R'}((\cs_I^*)_\mu )=\ann_{R'}((\cs_I^*)_0 )=H^0_\im (\sym_R (I))_0$ while
	for every  threshold integer $\mu$ it obtains
	$$
	\div ((\cs_I^*)_{\mu +1})=\div ((\cs_I^*)_\mu ).
	$$
 \end{prop}

\begin{proof}
	Let $\ell\in R_1$ not in any minimal prime of $I$. Then $\ell$ is a nonzero divisor on $\cs_I^*$. Hence for
	$\mu\geq 0$ the
	canonical inclusion $R_1\ann_{R'}((\cs_I^*)_\mu )\subseteq \ann_{R'}((\cs_I^*)_{\mu +1})$ is an equality. Furthermore,
	the exact sequence
	$$
	0\ra (\cs_I^*)_{\mu }\buildrel{\times \ell}\over{\lra} (\cs_I^*)_{\mu  +1}\ra (\cs_I^*/(\ell )\cs_I^*)_{\mu  +1}\ra 0
	$$
	shows that
	$$
	\div ((\cs_I^*)_{\mu +1})=\div ((\cs_I^*)_\mu )+\div ((\cs_I^*/(\ell )\cs_I^*)_{\mu  +1}).
	$$
	Let $\ol{R}:=R/(\ell )$,  $\ol{I}:=I/(\ell )\subset \ol{R}$, $H:=\proj (\ol{R})\subset \proj (R)={\bf P}^{n-1}$, notice that $\proj (\cs_I^*/(\ell )\cs_I^*)\subset H\times {\bf A}^{n+1}$
	coincides with $\proj (\sym_{\ol{R}}(\ol{I}))=\proj (\rees_{\ol{R}}(\ol{I}))$ and that $\ann_{R'}(\rees_{\ol{R}}(\ol{I}))_\mu )$
	is a prime ideal of  height two that does not depend on $\mu$ for any $\mu \geq 0$.

	It follows that
	$$\div ((\cs_I^*/(\ell )\cs_I^*)_{\mu  +1})=\div (\rees_{\ol{R}}(\ol{I})_{\mu +1})=0$$ if $H^0_\im ((\cs_I^*/(\ell )\cs_I^*))_{\mu  +1}=0$, which in turns
	hold if $H^1_\im (\cs_I^*)_{\mu }=0$. The conclusion then follows from Theorem 1.
\end{proof}

As a consequence of Proposition \ref{prop:divmu0} and Corollary \ref{cor:genres}, the matrix of the first map of the resolution of $\cs_I^*$,
in any basis with respect to the chosen degree, provides a matrix-based representation of the hypersurface $\ch$ if the base points, if any, are
locally complete intersection. Otherwise, if the base points are almost complete intersections, then some known extraneous factors appear; we
refer the interested reader to \cite{BCJ} for more details. We end this paper by summing up the consequence of the results presented in this paper
for the purpose of matrix-based representation of parameterized hypersurfaces.

\subsection{The $\im$-primary case.} This case is particularly comfortable because all the non-linear syzygies that appear in these matrix
representations can be computed by downgrading some linear syzygies of higher degree. This is a consequence of the isomorphisms given in Section \ref{downgrading}.

Recall that, as we explained just after Proposition \ref{prop2}, it is possible to tune the integer $\nu$ so that there is only linear and quadratic
syzygies in the matrix-based representation. Such a framework has been intensively studied by the community of Computer Aided Geometric Design
under the name ``moving surfaces method'' (see \cite {CGZ00,BCD03} and the references therein).

In the particular case $n=3$, we see that only linear and quadratic syzygies appear in the family of matrices $M_\mu$ with $\mu\geq \mu_0$.
If the $f_i$'s are in generic position, then $\mu_0=d-1$ and the matrix $M_{\mu_0}$ is a square matrix (all the $b_i$'s are equal to zero).
In the paper \cite{CGZ00}, a condition on the rank of the moving planes matrix is used. It is interesting to notice that it implies that $\mu_0=d-1$
and hence that the matrix $M_{\mu_0}$ is square. Indeed, with the notation of this paper, the condition in \cite{CGZ00} is $\dim (Z_1)_{2d-1}=d$. From the exact sequence
$$0 \rightarrow Z_1 \rightarrow R(-d)^4 \rightarrow R \rightarrow H_0 \rightarrow 0$$
we get $\dim (Z_1)_{2d-1}=d+\dim (H_0)_{2d-1}$. Therefore, the condition in \cite{CGZ00} implies that $\dim (H_0)_{2d-1}=0$.
Moreover, since $(B_1)_{2d-1}=0$ we have $\dim (H_1)_{2d-1}=d$ and the isomorphism $H_1\simeq {H_0}^*[3-4d]$ shows that
$\dim (H_0)_{2d-2}=d\neq 0$. Therefore, $\fin(H_0)=2d-2$ so that $\mu_0=d-1$ (as if the $f_i$'s were in generic position).

\subsection{In the presence of base points.} In this case, the downgrading maps are no longer available. So that the higher order syzygies
have to be computed as linear syzygies of a suitable power of the ideal $I$.

Notice that similarly to the $\im$-primary case, it is also possible to tune the integer $\nu$ in order to bound the order of the syzygies appearing
in the matrix representation. Mention also that if the ideal $I$ is saturated, so that $H^0_\im(H_1)=0$, it is remarkable that one never gets quadratic
syzygies in the first map of the complex. This is a direct consequence of Theorem \ref{thm:Fl/Fl-1free}.

\begin{expl} Take again Example \ref{HBexpl} and assume that $d=3$, that is to say that we start with a matrix of general linear forms ($e_i=1$ for
every  $i=1,2,3$). In this case, $\mu_0(I)=0$ and $\nu_0(I)=1$. The implicit equation, which is a cubic form, is then directly obtained in the case by
taking $\mu=0$ and is represented by a matrix of linear syzygies when $\mu\geq 1$. According to our previous observation, whatever $\mu\geq 0$ is,
there is no quadratic syzygies involved in the associated complex.
\end{expl}

\begin{expl}
We treat in detail the following example taken from \cite[Example 3.2]{BCD03}. All the computations have been done with the software Macaulay2 \cite{M2}.
$$f_0=X_0X_2^2, \ f_1=X_1^2(X_0+X_2), \ f_2=X_0X_1(X_0+X_2), \ f_3=X_1X_2(X_0+X_2)$$
and $d=3$. The ideal $I=(f_0,f_1,f_2,f_3)$ defines 6 base points: $(0:0:1)$, $(1:0:0)$ with multiplicity 2 and $(0:1:0)$ with multiplicity 3. Its saturation
$I^{\textrm sat}$ is the complete intersection $(X_0X_1+X_1X_2,X_0X_2^2)$, so that $\indeg(I^{\textrm sat})=2$.
The method developed in
\cite{BuCh05} shows that for every  $\mu\geq 2\times (3-1)-2=2$ one can obtain a matrix, filled exclusively with
linear syzygies, representing our
parameterized surface. For instance, such a matrix for $\mu=2$ is given by
$$\begin{pmatrix}{T_1}&
      0&
      0&
      0&
      0&
      {T_3}&
      0&
      0&
      0\\
      {-{T_2}}&
      {T_1}&
      0&
      0&
      0&
      0&
      {T_3}&
      0&
      {T_0}\\
      0&
      0&
      {T_1}&
      0&
      0&
      {-{T_2}}&
      0&
      {T_3}&
      0\\
      0&
      {-{T_2}}&
      0&
      {T_3}&
      0&
      0&
      0&
      0&
      0\\
      0&
      0&
      {-{T_2}}&
      {-{T_1}}&
      {T_3}&
      0&
      {-{T_2}}&
      0&
      {T_0}\\
      0&
      0&
      0&
      0&
      {-{T_1}}&
      0&
      0&
      {-{T_2}}&
      {-{T_2}}\\
      \end{pmatrix}$$
Although this is something that one does not want to do from a computational point of view, one can extract
from the above matrix the implicit
equation of our surface which is $T_1T_2T_3 + T_1T_2T_4 - T_3T_4^2=0$.

Now, we have $\mu_0(I)=2\times 2 - 4=0$ for $\indeg(H_1)=\indeg(H^0_\im(H_1))-d=4$. Moreover, since $H^0_\im(H_1)$
is concentrated in degree 7,
Theorem \ref{thm:Fl/Fl-1free} shows that in the case $\mu=0$ the matrix representing our surface is simply a
$1\times 1$-matrix whose entry is an implicit
 equation of the surface. However, the case $\mu=1$ is much more interesting because in this case the matrix
 representing the surface is filled with
 $\dim (Z_1)_{1+d}=3$ linear syzygies and $\dim H^0_\im(H_1)_{1+3d}=1$ quadratic syzygies since $\dim (R/I^{sat})_{-1}=0$.
 Here is this matrix
$$\begin{pmatrix}
T_2 & 0 & T_4 & -T_4^2 \\
-T_3 & T_4 & 0 & T_1T_3+T_1T_4 \\
0 & -T_2 & -T_4 & 0
\end{pmatrix}$$
It gives a representation of our parameterized surface. Notice that, as observed in \cite[Example 3.2]{BCD03},
 there does not exists a \emph{square} matrix of linear and/or quadratic syzygies whose determinant is an
 implicit equation of the surface.
\end{expl}

\appendix

\section{An argument of Joseph Oesterl\'e}

\medskip

\subsection{Un th\'eor\`eme \`a la Lefschetz}\label{Asec1}
Soient $n$ et $m$ deux entiers naturels. Consid\'erons
 l'anneau gradu\'e $${\rm R}={\bf Q}[x_1,\ldots,x_n]/(x_1^{m},\ldots,x_n^{m}),$$ o\`u les $x_i$ sont des
ind\'etermin\'ees. L'ensemble ${\rm R}_k$ de ses \'el\'ements homog\`enes de degr\'e $k$ est un
espace vectoriel de dimension finie sur ${\bf Q}$ pour tout  $k\in{\bf Z}$; il est non nul si et seulement si $0\le k\le d$, o\`u $d=n(m-1)$.

L'espace vectoriel ${\rm R}_{d}$ est de dimension $1$ sur ${\bf Q}$,  et $(x_1\ldots x_n)^{m-1}$ en est une base. Pour tout entier $k$ tel que $0\le k\le d$, l'application bilin\'eaire ${\rm R}_k\times {\rm R}_{d-k}\to{\rm R}_{d}$ induite par la multiplication de ${\rm R}$ est inversible~;  le rang de ${\rm R}_k$ est donc \'egal \`a celui de   ${\rm R}_{d-k}$.

Posons $\omega=x_1+\ldots+x_n$. Nous allons d\'emontrer le r\'esultat suivant:\medskip

\begin{thmf}
 Soient $k\in{\bf Z}$ et $t\in{\bf N}$. L'application ${\bf Q}$-lin\'eaire de ${\rm R}_k$ dans ${\rm R}_{k+t}$ induite par la multiplication par $\omega^t$ est injective si $2k+t\le d$, et
 surjective si $2k+t\ge d$.
\end{thmf}

Il nous suffira de d\'emontrer la premi\`ere assertion, car la seconde s'en d\'eduit par dualit\'e.
Nous  le ferons par r\'ecurrence sur $n$, en nous servant du lemme suivant, que nous
d\'emontrerons au ${\rm n}^{\circ}$\ref{Asec2}:

\begin{lemf}
	Soient ${\rm A}$ une ${\bf Q}$-alg\`ebre, $a$ un \'el\'ement de ${\rm A}$, $m$ et $t$ des entiers naturels,
	et $x$ une ind\'etermin\'ee. Pour qu'un \'el\'ement de ${\rm A}[x]/x^m{\rm A}[x]$ soit annul\'e par $(x+a)^t$, il faut et
	il suffit que ce soit la classe d'un polyn\^ome ${\rm P}(x)\in{\rm A}[x]$ de la forme
	$\sum_{j=1}^{\inf(m,t)} b_j{\rm P}_j(x)$,
	o\`u ${\rm P}_j(x)=\sum_{i=0}^{m-j}{(m+t-2j-i)!(j+i-1)!\over
	(m-j-i)!(j-1)!}(-a)^ix^{m-j-i}$ et o\`u  $b_j\in{\rm A}$ est annul\'e par $a^{m+t+1-2j}$.
\end{lemf}

\begin{remf}
	Comme  ${\rm P}_j(x)$ est de degr\'e $m-j$ et que son coefficient dominant est
	inversible, les $b_j$ dont il est question dans le lemme sont uniques.	
\end{remf}

Le th\'eor\`eme \'etant clair pour $n=0$, nous supposerons $n\ge1$. Nous appliquerons le lemme en
prenant pour ${\rm A}$ l'anneau ${\bf Q}[x_1,\ldots,x_{n-1}]/(x_1^{m},\ldots,x_{n-1}^{m})$, $a=x_1+\ldots+x_{n-1}$ et $x=x_n$, de sorte que ${\rm A}[x]/x^m{\rm A}[x]$ s'identifie  \`a ${\rm R}$ et que
$\omega=x+a$.

Soit $k$ un entier relatif tel que $2k+t\le d= n(m-1)$. Si un \'el\'ement de ${\rm R}_k$ est annul\'e par $\omega^t$, il est la classe d'un polyn\^ome ${\rm P}$ de la forme $\sum_{j=1}^{\inf(m,t)} b_j{\rm P}_j$, o\`u les ${\rm P}_j$ sont comme
dans le lemme et o\`u $b_j\in{\rm A}$ est annul\'e par $a^{m+t+1-2j}$.
Lorsqu'on munit ${\rm A}[x]$ de la graduation d\'eduite de celle de ${\bf Q}[x_1,\ldots,x_n]$,
${\rm P}_j$ est homog\`ene de degr\'e $m-j$. Vu l'assertion d'unicit\'e de la remarque ci-dessus,
les $b_j$ sont homog\`enes de degr\'e $k-m+j$. Comme
$$2(k-m+j)+(m+t+1-2j)=2k+t-m+1\le d-(m-1)=(n-1)(m-1),$$
l'hypoth\`ese de r\'ecurrence implique que les $b_j$ sont tous nuls et donc que ${\rm P}=0$. Cela d\'emontre le th\'eor\`eme.\bigskip

\subsection{D\'emonstration du lemme}\label{Asec2}

Nous adoptons les notations du lemme :  ${\rm A}$ est une ${\bf Q}$-alg\`ebre, $a$ est un \'el\'ement de ${\rm A}$, $m$ et $t$ sont des entiers naturels et $x$ est une ind\'etermin\'ee.

Notons ${\rm B}$ l'anneau ${\rm A}((x^{-1}))$ des s\'eries de Laurent en $x^{-1}$. Remarquons que
$x+a=x(1+ax^{-1})$ est un \'el\'ement inversible de ${\rm B}$, dont l'inverse est $\sum_{i=0}^\infty
(-a)^ix^{-i-1}$. Consid\'erons le sous-${\rm A}$-module
$${\rm E}={\rm A}x^{m}+{\rm A}x^{m+1}+\ldots+{\rm A}x^{m+t-1}$$ de ${\rm B}$; il  est libre de rang $t$. Notons
${\rm F}$ l'ensemble des  $f\in{\rm B}$ tels que $(x+a)^tf\in{\rm E}$. C'est un sous-${\rm A}$-module de ${\rm B}$ libre de rang $t$, puisque l'application $f\mapsto (x+a)^tf$ d\'efinit un isomorphisme de ${\rm F}$ sur ${\rm E}$. Consid\'erons les \'el\'ements $f_1,f_2,\ldots,f_{t}$ de ${\rm B}$ d\'efinis par :
$$f_j=\begin{cases}{({d\ \over dx})^{t-j}({x^{m+t-j}\over (x+a)^j})}&{\rm si} \ 1\le j\le \inf(m,t)\cr
{x^m\over (x+a)^{j}}&{\rm si} \ m+1\le j\le t.\cr\end{cases}$$

Il est clair que $(x+a)^tf_j$ est un polyn\^ome en $x$, que ce polyn\^ome appartient \`a $x^m{\rm A}[x]$,
que son degr\'e est $m+t-j$ et que son coefficient dominant est un entier naturel non nul
(\`a savoir ${(m+t-2j)!\over(m-j)!}$ si $1\le j\le \inf(m,t)$ et $1$ si  $m+1\le j\le t$). Il s'en suit que
les $(x+a)^tf_j$ forment une base du ${\rm A}$-module  ${\rm E}$ et que les $f_j$ forment une base du ${\rm A}$-module  ${\rm F}$.

Pour qu'un \'el\'ement de ${\rm A}[x]/x^m{\rm A}[x]$ soit annul\'e par $(x+a)^t$, il faut et
il suffit que le polyn\^ome ${\rm P}(x)\in{\rm A}[x]$ de degr\'e $\le m-1$ qui le repr\'esente appartienne \`a
${\rm F}$. Examinons donc \`a quelle condition une s\'erie de Laurent de la forme
$b_1f_1+\ldots+b_tf_t$, o\`u les $b_j$ appartiennent \`a ${\rm A}$, est un polyn\^ome.

Si $1\le j\le \inf(m,t)$,   la s\'erie de Laurent $f_j$ s'\'ecrit ${\rm P}_j+a^{m+t-2j+1}g_j$, o\`u
${\rm P}_j\in{\rm A}[x]$ est le polyn\^ome en $x$ figurant dans le lemme et $g_j$ est une s\'erie formelle
en $x^{-1}$ sans terme constant dont le terme de plus bas degr\'e en $x^{-1}$ est le produit
d'un entier relatif non nul par $x^{-(t-j+1)}$.
Si $m+1\le j\le t$, on pose $g_j=f_j$: dans ce cas
 $g_j$ est une s\'erie formelle en $x^{-1}$ sans terme constant
dont le terme de plus bas degr\'e est  $x^{-(j-m)}$. De ces propri\'et\'es, on d\'eduit que les s\'eries formelles
$g_1,\ldots,g_t$ sont lin\'eairement ind\'ependantes sur~${\rm A}$. Il s'en suit que
$b_1f_1+\ldots+b_tf_t$ est un polyn\^ome si et seulement si on a $b_ja^{m+t-2j+1}=0$
pour $1\le j\le \inf(m,t)$ et $b_j=0$ pour $m+1\le j\le t$. Et ce polyn\^ome est alors
\'egal \`a $\sum_{1\le j\le \inf(m,t)}b_j{\rm P}_j$, d'o\`u le lemme.


\begin{thebibliography}{RRR91}

\bibitem[AH80]{AH}
L{\^a}cezar Avramov and J{\"u}rgen Herzog.
\newblock The koszul algebra of a codimension 2 embedding.
\newblock {\em Math. Z.}, 175 (1980), no. 3, 249--260.

\bibitem[BC05]{BuCh05}
Laurent Bus{\'e} and Marc Chardin.
\newblock Implicitizing rational hypersurfaces using approximation complexes.
\newblock {\em J. Symbolic Comput.}, 40 (2005), no. 4-5, 1150--1168.

\bibitem[BCD03]{BCD03}
Laurent Bus{\'e}, David Cox, and Carlos D'Andrea.
\newblock Implicitization of surfaces in {${\Bbb P}\sp 3$} in the presence of
  base points.
\newblock {\em J. Algebra Appl.}, 2 (2003), no. 2, 189--214.

\bibitem[BCJ09]{BCJ}
Laurent Bus{\'e}, Marc Chardin, and {Jean-Pierre} Jouanolou.
\newblock Torsion of the symmetric algebra and implicitization.
\newblock {\em Proc. Amer. Math. Soc.},
  137 (2009), no. 6, 1855--1865.

\bibitem[BJ03]{BuJo03}
Laurent Bus{\'e} and Jean-Pierre Jouanolou.
\newblock On the closed image of a rational map and the implicitization
  problem.
\newblock {\em J. Algebra}, 265 (2003), no. 1, 312--357.

\bibitem[CGZ00]{CGZ00}
David Cox, Ronald Goldman, and Ming Zhang.
\newblock On the validity of implicitization by moving quadrics of rational
  surfaces with no base points.
\newblock {\em J. Symbolic Comput.}, 29 (2000), no. 3, 419--440.

\bibitem[GS]{M2}
Daniel~R. Grayson and Michael~E. Stillman.
\newblock Macaulay 2, a software system for research in algebraic geometry.
\newblock Available at http://www.math.uiuc.edu/Macaulay2/.

\bibitem[HSV83]{HeSV83}
J.~Herzog, A.~Simis, and W.~V. Vasconcelos.
\newblock Koszul homology and blowing-up rings.
\newblock In {\em Commutative algebra ({T}rento, 1981)}, pp. 79--169, Lecture Notes in Pure and Appl. Math., 84, Dekker, New York, 1983. 

\bibitem[HSV09]{HSV09}
J.~Hong, A.~Simis, and W.~Vasconcelos.
\newblock The equations of almost complete intersections.
\newblock Preprint, 2009.

\bibitem[RRR91]{RRR91}
Les Reid, Leslie~G. Roberts, and Moshe Roitman.
\newblock On complete intersections and their {H}ilbert functions.
\newblock  Canad. Math. Bull. 34 (1991), no. 4, 525--535.

\end{thebibliography}

\end{document}